\documentclass[12pt]{article}
\usepackage{amsmath,amsthm,amssymb,color,amscd}
\usepackage{color}
\usepackage{enumitem}
\usepackage{soul}
\usepackage{ulem}
\usepackage{cancel}

\definecolor{dgreen}{rgb}{0,0.5,0}
\definecolor{brown}{rgb}{0.5,0,0}
\definecolor{magenta}{rgb}{1,0,1}
\definecolor{orange}{rgb}{1,0.4,0}

\def\seq#1#2#3{#1_{#2},\,\ldots,#1_{#3}}

\def\w{\widetilde}

\def\k{\overline}

\def\st{\;|\;}

\def\betab{\k{\beta}}

\def\1{\underline{1}}
\def\ord{\text{ord}}

\def\P{\mathbb P}
\def\N{\mathbb N}

\def\Z{\mathbb Z}
\def\Q{\mathbb Q}
\def\C{\mathbb C}

\def\OO{{\mathcal O}}

\def\DD{{\mathcal D}}

\def\PP{{\mathcal P}}

\def\BB{{\mathcal B}}
\def\CC{{\mathcal C}}
\def\JJ{{\mathcal J}}
\def\SS{{\mathcal S}}
\def\TT{{\mathcal T}}

\newcommand{\pr}{\textstyle\prod\limits}

\newtheorem{theorem}{Theorem}[section]
\newtheorem{statement}[theorem]{Statement}
\newtheorem{corollary}[theorem]{Corollary}
\newtheorem{lemma}[theorem]{Lemma}
\newtheorem{proposition}[theorem]{Proposition}

\theoremstyle{definition}
\newtheorem{definition}[theorem]{Definition}

\newtheorem{remark}[theorem]{Remark}
\newtheorem{example}[theorem]{Example}
\newtheorem{setting}[theorem]{}

\title{On the image of a curve in a normal surface by a  plane projection
\footnote{Math. Subject Class. 14H20, 32S05, 32S15, 32S45,
32S55}}
\author{F.~Delgado \thanks{
Partially supported by grant PID2022-138906NB-C21 funded by
MCIN/AEI/ 10.13039/501100011033 and by ERDF "A way of making Europe".
The author is thankful to the Institut Fourier, Universit\'e de
Grenoble I for hospitality.\newline
Address:
IMUVA (Instituto de Investigaci\'on en Matem\'aticas).
University of Valladolid. Valladolid, Spain.
E-mail: fdelgado\symbol{'100}uva.es.
}
\and H. Maugendre \thanks{Partially supported by the ANR LISA Project ANR-17-CE40-0023.\newline
Address:
Institut Fourier,
Universit\'{e} de Grenoble I, BP 74, F-38402 Saint-Martin
d'H\`{e}res, France.  E-mail: helene.maugendre\symbol{'100}univ-grenoble-alpes.fr}
}
\date{}
\begin{document}
\sloppy
\def\eps{\varepsilon}

\maketitle

\begin{abstract}
We consider a finite analytic morphism $\varphi =(f,g)$ defined from a complex
analytic normal surface $(Z,z)$  to $\C^2$.
We describe the topology of the image by $\varphi$ of a  reduced curve  on
$(Z,z)$ by means of iterated pencils defined recursively for each branch of the curve from the initial one
$\langle f,g \rangle$. This result generalizes the one obtained in a previous paper
for the case in which
$(Z,z)$ is smooth and the curve irreducible.
As a consequence of the methods we can describe also the topological
type of the
discriminant curve of $\varphi$, in particular the topological type of each branch of
the discriminant can be obtained from the map without the previous knowledge of the
critical locus.
\end{abstract}


\section*{Introduction}

Let $(Z,z)$ be the germ of a complex analytic normal surface and let
$\varphi = (f,g): (Z,z)\to (\C^2,0)$ be the germ of a
finite complex analytic morphism.
Let $\gamma\subset (Z,z)$ be the germ of a reduced curve,  $\gamma
=\bigcup_{i=1}^r \gamma_i$, with $\gamma_i$ a branch in $(Z,z)$, and
$\delta =\bigcup_{i=1}^r \delta_i$ with $\varphi(\gamma_i)=\delta_i,
1\leq i \leq r$. If $(Z,z)$ is equal to $(\C^2,0)$ and $\gamma $ is irreducible,
in \cite{TopIm} we have described the topology of the plane branch $\delta$ by using
iterated pencils of functions defined recursively from the initial one $\Phi= \langle
f,g\rangle$. The purpose of this article is to show that this result can be
generalized to the case of normal surfaces and when the curve $\gamma$
is not necessarily
irreducible; it means that we can describe the topological type of each branch
$\delta_i, 1\leq i\leq r$ (section 1), and moreover   give the intersection
multiplicities  between each pair of branches $\delta_i,\delta_j$ (section
2), which completely describe the topology of $\delta$.

In section 1, we show that, as in the plane  case, the topology of
the image by $\varphi $ of an irreducible curve $\gamma$  can be determined by the
degree of the restriction map $\varphi|_{\gamma}: \gamma\to
\delta=\varphi(\gamma)$
and intersection multiplicities between $\gamma$ and some particular fibres
of a sequence of iterated pencils constructed recursively from
$\langle f,g\rangle$ (see \ref{pencils}).
This result is stated in Theorem
\ref{thm1}. In section
2 we treat the
case when $\gamma$ is  reduced but no more  irreducible. It is well-known that the
topological type of its plane curve image by $\varphi$, $\delta =\bigcup_{i=1}^r
\delta_i$, is enterely determined by the topological type of each irreducible
component $\delta_i=\varphi (\gamma _i)$ and the intersection multiplicity between
each pair $(\delta_i, \delta_j) , j\neq i$. As the topological type of each
irreducible component  image $\delta_i$ has been treated in section 1, it leaves to
compute the intersection multiplicity between each pair of branches of $\delta$.
This result is stated in Theorem \ref{TheoremPrincipal}.
The proof consists in first computing such intersection multiplicities when
$\varphi$ is  the identity map
from the plane to the plane (Theorem
\ref{mainth-plane}),
then using the projection formula for the intersection multiplicity
(see Proposition
\ref{proj-for}) and Theorem \ref{mainth-plane} we can compute the
intersection multiplicity between $\delta _i$ and $\delta _j$ for the
general case of a finite map $\varphi: (Z,z)\to (\C^2,0)$.

A remarkable fact is that, for the computation of the intersection multiplicity, one
needs more precise informations than for the topological type of the branches (see
the Example \ref{ex1}). Namely
informations coming from the branches of the fibres and not only the contacts of
the branches of $\gamma$ with the whole fibres.

In  \cite{TopIm} the process
was in particular  applied to the description of the
branches of the discriminant locus
of $\varphi: (\C^2,0)\to (\C^2,0)$.
To do that the key point was the study of the special fibres of the pencil
$\langle f,g \rangle$
and
their relations with the critical locus of $\varphi$ developed in \cite{DM} as
well as the relation with the behaviour of the Hironaka quotients established in
\cite{LMW} and \cite{Mi}.
For a normal singularity $(Z,z)$ and a finite map $\varphi=(f,g): (Z,z)\to(\C^2,0)$
the relation of the special fibres with the critical locus was generalized in
\cite{PNS}. (For a reduced surface singularity the behaviour of the generic and
special fibres has been recently treated in \cite{BM-JS}).
These relations allow us to extend the results of \cite{TopIm} and so apply the
results of Sections 1 and 2 to describe the topological type of the
discriminant curve of $\varphi$. This is the purpose of Section 3.
In relation with the study of the discriminant curve one has to mention the recent
paper \cite{GB-PP} from Garc\'{\i}a-Barroso and Popescu-Pampu in which they show that
the Newton polynomial of  the discriminant curve essentially depends only on the
pencil $\langle f,g\rangle$ (see also \cite{GGP} for the plane case).
Althougth the results are different, some of the technics used there are very close
to the ones we use here.

\section{Case of an irreducible curve}
Let $\gamma\subset (Z,z)$ be an irreducible curve (in short, a branch) and take
$\eta : (\C, 0) \to (\gamma , z)$ a parametrization (uniformization) of
$(\gamma, z)$. Let
$\{\ell=0\}\subset Z$ be the curve defined by the analytic function $\ell : (Z,z)\to
(\C,0)$. We
define the intersection multiplicity of $\{\ell=0\}$ and $\gamma$ at $z$, denoted $I_z(\ell,
\gamma)$,    by:
$$I_z(\ell, \gamma)=\  \mbox{ord} \ _t(\ell\circ \eta (t)).$$
Notice that the normalization of the local ring $\OO_{\gamma,z}$ is isomorphic to $\C\{t\}$ and 
$\nu_{\gamma}(-):=I_z(-, \gamma)$ is the valuation $\nu_\gamma$
defined by $\gamma$. Moreover, we write $I_z(\ell,\gamma)=\infty$ if $\gamma$ is a branch of
$\{\ell=0\}$, i.e. if the function $\ell$ vanishes on $\gamma$.
If $(Z,z)$ is smooth then $I_z(\ell,\gamma)$ coincides with the usual
intersection multiplicity
between the curve germs $L:=\{\ell=0\}$ and
$\gamma$. In this case we use also the notation $[L,\gamma]=[\{\ell=0\},\gamma]$ (and $[L,\gamma]_z$ if we want to specify the point) to
indicate the intersection multiplicity at $z\in Z$.

\medskip

Let $\varphi^*: \OO_{\C^2,0}\to \OO_{Z,z}$,
be the ring homomorphism associated to
$\varphi$, i.e $h\mapsto \varphi^* h=h\circ \varphi$ for $h\in \OO_{\C^2,0}$.
If $\xi = \{h=0\}\subset (\C^2,0)$ we use also $\varphi^*\xi$ to denote
the local curve in $(Z,z)$ defined by $\varphi^* h$. Notice that $\varphi^*\xi$ is a
local Cartier divisor in $(Z,z)$.

Let $\gamma\subset (Z,z)$ be a branch and $\varphi(\gamma)=\delta$.
The restriction map $\varphi|_{\gamma}: \gamma\to \delta$ has degree $k\ge 1$ and
we define the {\it direct image} of $\gamma$ as $\varphi_* \gamma := k\cdot \delta$.

\begin{proposition}[Projection formula]\label{proj-for}
One has:
\begin{equation}\label{eq1}
I_z(\varphi^* h, \gamma) =
I_0(h, \varphi_* \gamma) (=
k I_0(h,\delta) )
\end{equation}
\end{proposition}

\begin{proof}
Let ${\mathfrak p}$ be the prime ideal on $\OO_{Z,z}$ defining the curve germ
$\gamma$ on
$(Z,z)$, then the prime ideal ${\mathfrak q} = (\varphi^*)^{-1}({\mathfrak p})$ is
the ideal
defining $\delta =\varphi (\gamma)\subset (\C^2,0)$ and one has the finite extension of
one-dimensional local rings
$\OO_{\delta,0} = \OO_{\C^2,0}/\mathfrak{q}\hookrightarrow
\OO_{\gamma,0} = \OO_{Z,z}/\mathfrak{p}$. The degree of the above extension is just
$k$: the one of the restriction map $\varphi|_{\gamma}: \gamma\to \delta$.
Let $\C\{t\}$ (respectively $\C\{\tau\}$) be isomorphic to the normalization of
$\OO_{\gamma,z}$ (respectively to the one of
$\OO_{\delta,0}$), one has $\C\{\tau\}\subset \C\{t\}$ and $\tau(t)\in \C\{t\}$ has
order $k$. In fact one could take $\tau = t^k$ with a convenient choice of the
uniformizing parameters $\tau$ and $t$.
Now, let $h \in \OO_{\C^2,0}$ be an holomorphic  function on $\C^2$. One has that
$\ord_t(\varphi^* h (t)) = \ord_t(h(t^k)) = k \cdot \ord_{\tau}(h(\tau))$.
\end{proof}

\medskip
\begin{remark}
The above
formula~(\ref{eq1})
extends the classical projection formula for a map $\varphi$ between smooth surfaces
(see, e.g. \cite{fulton}). Notice also that one could extend the above formula for
any local Weil divisor
$\gamma = \sum_{i=1}^{\ell}n_i \gamma_i$ on $(Z,z)$.
Finally, the formula (\ref{eq1}) is obviously also true for a finite morphism
$\varphi: (Z,z)\to (X,x)$ between normal surface singularities.
\end{remark}

\begin{setting}
Let $\pi: (X,E)\to (Z,z)$ be a good resolution of $(Z,z)$, i.e. a resolution of the
singularity $(Z,z)$ such that the exceptional divisor
$E=\bigcup_{\sigma\in \Gamma}E_\sigma$ is a union of smooth projective curves with
normal crossings.
 For $\beta\in \Gamma$ and for each
holomorphic function $h: (\C^2,0)\to (\C,0)$ let $\nu_{\beta}(h)$ denote the
vanishing order of
$\k h = h\circ \pi: X\to \C$ along  the irreducible exceptional curve
$E_{\beta}$ ($\nu_{\beta}$ is just the divisorial valuation defined by
$E_{\beta}$). 
If $\gamma = \{ h=0\}$ is the curve defined by $h$ we use also the notation $\nu(\gamma)$ instead $\nu (h)$, in particular this notation could be used for any curve germ if $(Z,z) = (\C^2,0)$ and $\pi$ is a modification of it.

The lifting $\k \psi = \psi\circ \pi$ of the meromorphic function $\psi=g/f$ is a
meromorphic function defined in a suitable neighbourhood of the exceptional divisor
$E$ but in a finite number of points. The irreducible component $E_\beta$ of $E$ is
called a
{\it dicritical\/} component of the pencil $\Phi=\langle f,g \rangle$ (or of the
meromorphic
function $\psi$) if the restriction of $\k \psi : E_\beta\to \C\P^1$ is everywhere defined and not
constant.

A good resolution of the pencil $\Phi$ is a good resolution $\pi: (X,E)\to (Z,z)$ of $(Z,z)$ such that $\k\psi$ is everywhere defined. In such a case, the zero locus $\{h=0\}$ of the elements $h\in \Phi$ are equisingular and called the generic fibres of $\Phi$, excepted a finite number of them called the special fibres. Moreover from \cite{PNS}, we know that the strict transforms of the generic fibres by $\pi$ intersect the exceptional divisor smoothly and transversaly at smooth points of the dicritical components. In fact, a good resolution of $(Z,z)$ is a good resolution of $\Phi$ if and only if is a resolution of all the fibers of $\Phi$ but a finite number.
\end{setting}

\begin{lemma}\label{lemmaInterDiv}
With the above notations.
Let $h:(Z,z)\to (\C,0)$ be an analytic function, $\eta=\{h=0\}$ and let $(\gamma,z)\subset
(Z,z)$ be a branch such that its strict transform  by $\pi$, $\w{\gamma}$, intersects
$E_\alpha\subset E$ at a smooth point $P$.
Then
$$
I_z(h, \gamma)= [\w \eta, \w \gamma]_P +[E_\alpha,\w \gamma]_P \; \nu_\alpha(h)\; .
$$
In particular, if 
$\tilde \eta \cap \tilde \gamma =\emptyset$, then
$
I_z(h,\gamma)=[E_\alpha, \tilde \gamma]_P \; \nu_\alpha (h)\; .
$
\end{lemma}

\begin{proof}
The divisor $(\k h)$
 defined by $\k h = h\circ \sigma$ on $X$ could be written as
$$
(\k h) = (\w h) + \sum_{\beta\in \Gamma}\nu_{\beta}(h) E_{\beta} \; ,
$$
where the local part $(\w h)= \w \eta$ is the strict transform of the germ $\{h=0\}$
by $\pi$.
For each $\sigma\in \Gamma$ one has the known Mumford formula (see \cite{Mu}):
\begin{equation}
\label{mumford}
(\k h)\cdot E_{\sigma} = (\w h)\cdot E_{\sigma} + \sum_{\beta} \nu_{\beta}(h)
(E_{\beta}\cdot E_{\sigma}) = 0\; .
\end{equation}
(Here ``$\cdot$" stands for the intersection form on the smooth surface $X$).

Applying the equation (\ref{mumford}) to the total transforms of $\gamma$ and $\eta$ we have:
$$
I_z(h, \gamma)
= (\k \eta)\cdot (\k \gamma) =
(\w h)\cdot \w \gamma +
\sum_{\beta\in \Gamma} \nu_{\beta}(h) (E_\beta\cdot \w \gamma ) =
[\w \eta, \w \gamma]_P + \nu_\alpha(h) [E_\alpha, \w \gamma]_P\; .
$$
\end{proof}

The above result will be used  in the paper also in the case of
$(Z,z)=(\C^2,0)$. Note that in this case every curve is defined by a function (i.e. is a
Cartier divisor).

\begin{setting}[{\bf Iterated pencils}]\label{pencils}
We denote $\Phi = \langle g, f \rangle= \{ a g - b f : a, b
\in  \C\}$
the pencil  defined by $f$ and $g$.
Let us assume that
$I_z(g,\gamma)\ge I_z(f,\gamma)$.
If $I_z(g,\gamma)=\infty$ then $\varphi(\gamma)$ is the $x$ axis $\{y=0\}$ and the
procedure stops here. Otherwise,  let $q/p = I_z(g,\gamma )/I_z(f,\gamma )\ge 1$ with gcd$(p,q)=1$.
Let $\Phi_1$ be the pencil
generated by
$g^p$ and
$f^q$. From Proposition 1 of \cite{PNS} the intersection multiplicity
$I_z(ag^p-bf^q,
\gamma)$ is constant (and equal to 1) for all $w=(a:b)\in \C\P^1$,
 except for  one value $w_0$ in $\C\P^1$. With the hypothesis that  $q/p\ge 1$,
it is obvious that we can write
$w_0= (1:a_1)$. We denote  $g_1= g^p-a_1f^q$.

Let $f_1:=f^q$, and $\varphi
_1= (f_1,g_1) : (Z,z) \longrightarrow (\C^2, 0)$.
If $I_z( g_1,\gamma)=\infty$ then $\gamma$ is a
branch of $\{ g_1=0\}$ and in  this case one has  $\varphi _1(\gamma)=L$,
$L= \{y_1=0\}$ for $y_1 = y^p-a_1 x^q$, so
$\varphi(\gamma)=\{ y^p-a_1x^q=0\} $
and we stop here.

Else, we consider
$\Phi_1 = \langle g_1, f_1\rangle$ and we can define a new
rational number
${q_1}/{p_1} = I_z(g_1,\gamma )/I_z(f_1,\gamma)>1$.
In this case we can proceed to construct $\Phi_2 = \langle
g_2, f_2\rangle$, $g_2\in \Phi_2$,
 and ${q_2}/{p_2}$
 as we did for $\Phi_1, g_1, q_1/p_1$.

Recursively, we  define a sequence
$\PP(f,g,\gamma) = \{(\Phi_i,g_i,q_i/p_i) : i\ge 0\}$ where,
$(\Phi_0, g_0, q_0/p_0) = (\Phi, g, q/p)$ and, for $i\ge 1$,
$\Phi_i = \langle g_i,f_i \rangle =
\langle g_{i-1}^{p_{i-1}},f_{i-1}^{q_{i-1}} \rangle$ is a pencil of functions on
$(Z,z)$, $g_i$ is the unique function of
$\Phi_i$ such that $I_z(g_i, \gamma)>I_z(h,\gamma)$ ($h\in \Phi_i$ a generic
function) and
${q_i}/{p_i} = I_z(g_i,\gamma)/I_z(f_i,\gamma)>1$, with $\gcd(q_i,p_i)=1$.
We denote also $g_{-1}=f$ and by $\xi_i$ the curve on $(Z,z)$ defined by
$g_i=0$.
Note that, for $i\ge 1$, $f_i = f^{q_0 \cdots q_{i-1}}$ and also that  the sequence
is infinite unless there exists $l\in \N$ such that
$q_l/p_l=\infty$ in which case the sequence is finite and ends at $l$.

\end{setting}

\begin{definition}
The {\it iterated sequence of pencils } of the branch $\gamma$ (with respect
to $\varphi=(f,g)$) is the sequence $\PP(f,g,\gamma)$  defined above.
\end{definition}

\begin{remark}\label{pencils-plane}
Along the paper we use frequently the case
$(Z,z) = (\C^2,0)$ and $\varphi =(x,y)$ the identity map from
$(\C^2,0)$ to $(\C^2,0)$. In this case we use
$\{(\Lambda_i,y_i, q_i/p_i) : i\ge 0\}$ to denote the iterated pencil of a branch
$\delta\subset (\C^2,0)$. We will also denote $\lambda_i$ the curve defined by
$y_i=0$.

\end{remark}

\begin{theorem}\label{thm1}
Let $k$ be the degree of the
restriction map $\varphi _{|\gamma} : \gamma \longrightarrow \delta$ and
let $\varphi_*  \gamma = k\cdot \delta$ be the direct image
of $\gamma$.
The sequence
$\{I_z(g_i,\gamma ) : i= -1, 0, \ldots \}$ and the integer $k$
determine the topological type of $\delta$.
\end{theorem}

And as an easy consequence one has:

\begin{corollary}\label{cor1} The sequence of rational numbers
$\{ {q_i}/{p_i}, i\geq 0\}$ determines the topological type of
$\delta$.
\end{corollary}

\begin{proof}
The proof of Theorem \ref{thm1} and Corollary \ref{cor1} repeats the scheme used in
the Theorem 1 (see also Corollary 1) of
\cite{TopIm}. The first step consists
in the proof that the topological type (say e.g. the semigroup of values) of $\delta$
could be determined with the described procedure in the particular case
in which we take the identity map
$(x,y): (\C^2,0)\to (\C^2,0)$.
For the sake of completeness we will describe the procedure, the proof could be read in
\cite{TopIm}.

Assume that
$\delta$ is not a coordinate axis and
$I_0 (x,\delta )\le I_0(y,\delta )$. Let us consider
the iterated sequence of pencils $\PP(x,y,\delta)=\{(\Lambda_i,y_i,q_i/p_i):i\ge 0\}$
of $\delta$ with respect to the identity.

Let $m_{-1}, m_0, \ldots$ be the sequence of positive integers
defined recursively as follows:
$m_{-1}=I_0(x,\delta)$,
$m_{0}=I_0(y,\delta)$ and, if we assume that $\seq m1i$ have
been defined, then
\begin{equation}\label{eq1-1}
m_{i+1} := d_i m_i + I_0( y_{i+1}, \delta ) - p_i I_0(y_i, \delta )\; ,
\end{equation}
where $d_i= \gcd (m_{-1}, \ldots, m_{i-1})/ \gcd (m_{-1}, \ldots,
m_{i})$.

Theorem 2 in \cite{TopIm} says that the
sequence $m_{-1}, m_0, \ldots$ generates the
semigroup of values of $\delta$ (see also the Remark below). 

\medskip

Now, using  the ring homomorphism $\varphi^*: \OO_{\C^2,0}\to \OO_{Z,z}$  associated
to $\varphi$, one can recover by pullback the iterated sequence of pencils
$\PP(f,g,\gamma)$ on $(Z,z)$: that means
$\varphi^*(\Lambda_i) = \Phi_i$, $\varphi^*(y_i)=g_i$ for $i\ge 0$.
The projection formula \ref{proj-for} applied to the equations \ref{eq1-1}
gives:
$$
m_{-1}=I_0(x,\delta)= I_z(f, \gamma)/k, \quad
m_{0}=I_0(y,\delta)=I_z(g,\gamma)/k
$$
and for $i\ge 1$:
$$
m_{i+1}= d_i m_i + I_0( y_{i+1}, \delta ) - p_i I_0(y_i, \delta )= d_i m_i
+\displaystyle\frac{ I_z(g_{i+1},\gamma)}{k }- p_i
\displaystyle\frac{I_z(g_i,\gamma)}{k}\; .
$$
Thus, we recover the semigroup of values of the branch $\delta$.

\medskip
The proof of the Corollary is the same as the one of Corollary 1 in \cite{TopIm}, however is also included in the proof of Theorem \ref{Thm1-3}.
\end{proof}

\begin{remark}\label{rmk_m}
In the proof of Theorem \ref{thm1} we recover the semigroup of values, in the same
way one can recover the characteristic exponents of $\delta$.
Let $\mu_i  = I_z(g_i,\gamma)$ for  $i\ge -1$.
Let $m_{-1} = \w m_{-1} =  I_z(f,\gamma)/k$, $m_0 = \w m_0 = I_z(g_0,\gamma)/k$ and
for $i\ge 1$,
\begin{align*}
m_{i+1} & = d_i m_i + \frac{\mu_{i+1}}{k} - p_i
\frac{\mu_i}{k} = d_i m_i + \frac{\mu_{i+1}}{k} - q_i
\frac{I_z(f_i,\gamma)}{k}\\
\w m_{i+1} & = \w m_i + \frac{\mu_{i+1}}{k} - p_i
\frac{\mu_i}{k} =
\w m_i +  \frac{\mu_{i+1}}{k} -
q_i  \displaystyle\frac{I_z(f_i,\gamma)}{k}
\end{align*}
where, $d_i = \gcd(\seq{m}{-1}{i-1})/\gcd(\seq{m}{-1}{i})$. Note that
$$
e_i:=\gcd(\seq{m}{-1}i) = \gcd(\seq{\w m}{-1}i)= \gcd(\seq{\mu}{-1}i)/k\; .
$$
As a consequence, denoting $\epsilon_i:=\gcd(\seq{\mu}{-1}{i})$, one has  also
$d_i =
\epsilon_{i-1}/\epsilon_{i}$.

From \cite{TopIm} one knows that the minimal set of generators
$\{\seq{\betab}0g\}$ of the semigroup of the plane germ $\delta$ is a subset
of $\{m_i \st i\ge -1\}$. More specifically
$\{\seq{\betab}0g\} = \{m_{-1}\}\cup \{m_i : d_i>1\}$.
In the same way, the set of characteristic exponents
$\{\seq{\beta}0g\}$ of $\delta$ is a subset of $\{\w m_i\st i\ge -1\}$, namelly the
set $\{\seq{\beta}0g\} = \{\w m_{-1}\}\cup \{\w m_i : d_i>1\}$.

In particular the computation of the sequences stops in a finite number of steps, just
when $e_i=1$ or equivalently $\gcd(\seq{\mu}{-1}i)=k$.
Notice that $k =\gcd \{\mu_i : i\ge -1\}$, i.e. the set $\mu=\{\mu_i:i\ge -1\}$ generates a finitely generated subsemigroup of $k\Z$. So theoretically one can recover the integer $k$ and so the topological type of $\delta$ from $\mu$, however it can not be computed in practice.  
\end{remark}

\section{Case of several branches}

In this section we treat the case in which $\gamma \subset (Z,z)$ is a reduced but
not irreducible curve in $(Z,z)$. We denote $\gamma =\bigcup_{i=1}^r
\gamma_i$, with $\gamma_i$ a branch in $(Z,z)$, and its image by $\varphi$ is
$\delta =\bigcup_{i=1}^r \delta_i$ in such a way that
$\varphi(\gamma_i)=\delta_i, 1\leq i \leq r$. Pay attention that we could have
$\delta_i=\delta_j$ for some $i\neq j$.

To describe the topological type of $ \delta$ it suffices to describe the
topological type of each branch $\delta_i, 1\leq i\leq r$, and give the
intersection multiplicities $[\delta_i,\delta_j]$ between pairs of different
branches $\delta_i,\delta_j$. Thus it is enough to resolve the case of two branches.

Along the section $\gamma, \gamma'\subset(Z,z)$ are two irreducible curves in $(Z,z)$
and we will assume
that $\varphi(\gamma)=\delta$ is not equal to  $\varphi(\gamma ')=\delta'$.

\subsection{Case of $(\C^2, 0)$}

We will
start, as in the irreducible case, describing the intersection multiplicity
$[\delta, \delta ']$ by the use of the iterated pencils starting with the pencil $\Lambda_0=\langle x,y\rangle$ on $\C^2$
corresponding to the identity map from 
$(\C^2,0)$ to $(\C^2,0)$.
So, we proceed by induction.

\subsubsection{First step}\label{first step}
Let assume that $I_0(\delta,y)\ge I_0(\delta,x)$ and let ${I_0(\delta,y)}/{
I_0(\delta,x)}={q}/{p} \geq 1$, with $\gcd(p,q)=1$.
We have to consider several situations according to the value of
${I_0(\delta',y)}/{ I_0(\delta',x)}$:

\bigskip
{\bf a)}
If ${I_0(\delta',y)}<{ I_0(\delta',x)}$ then
$\{ x=0\}$ is tangent to $\delta '$ and
transversal to $\delta$. So $\delta$ and $\delta'$ are transversal and we have
$[\delta, \delta ']= I_0(\delta, x)I_0(\delta', y)$.

\medskip

Otherwise we have  ${I_0(\delta',y)}/{ I_0(\delta',x)}= q'/p'\geq 1$,  with
$\gcd(p',q')=1$.

\medskip
{\bf b)}
Let us suppose that $q'/p'\neq q/p$, and $q'/p'\ge 1 $ and $ q/p \ge 1$. 
Assume that $q'/p'>q/p\geq 1$.

\noindent Note that if $q/p=1$ then $\delta $ and $\delta '$ are transversal and then
$$
[\delta, \delta ']= I_0(\delta, x)I_0(\delta', x)
= I_0(\delta, y)I_0(\delta', x) < 
I_0(\delta', y)I_0(\delta, x) \; . 
$$
Else we have $q'/p'>q/p> 1$ and in this case $\{ y=0\}$ is tangent to $\delta $ and
$\delta'$. Then, an easy computation (see also the proof of Proposition \ref{IntersectionMult}) allows to show that \\
$$[\delta, \delta ']=\min \{  I_0(\delta, x)I_0(\delta', y),  I_0(\delta,
y)I_0(\delta', x)\} =I_0(\delta, y)I_0(\delta', x)\; .$$

\medskip
{\bf c)} Endly, let us assume that $q/p=q'/p' \geq 1$.
Let $\Lambda _1=\langle y^p , x^q\rangle$ be the first
iterated pencil. Notice that $\Lambda_1=\Lambda_1'$. Taking into
account the construction of (\ref{pencils}), let $y_1 =y^p-ax^q \in
\Lambda
_1$, $a\neq 0$, be the only  fibre of $\Lambda_1 $ such that  $ I_0(\delta, y_1)>I_0(\delta ,y^p)
(=I_0(\delta, x^q))$ (resp. $y_1'=y^p-a'x^q \in \Lambda_1$, $a'\neq 0$, the corresponding one
for $\delta '$). We have to distinguish two cases according to $a=a'$ or $a\neq a'$.

\medskip
{\bf c-1)}\underline{\it Case  $a\neq a'$.}

\begin{lemma}
If $a\neq a'$ then $[\delta, \delta '] =  I_0(\delta, x)I_0(\delta', y)=
I_0(\delta, y)I_0(\delta', x)$.
\end{lemma}

\begin{proof} Let $\sigma : (X,E)\to (\C^2,0)$ be the minimal resolution of
$\Lambda_1$. Notice that the fibers $y_1,y_1'$ are generic for $\Lambda_1=\langle
y^p,x^q\rangle$, in particular they are equisingular.  The fact
that $  I_0(\delta, y_1)>I_0(\delta, y^p)$ implies that the strict transform of
$\delta$ and $y_1$ by $\sigma$ intersect in a point $P\in E_\alpha \subset E$. As
$y_1$ is generic for $\Lambda_1$ it implies  that $E_\alpha$ is a dicritical divisor 
for $\Lambda_1$
(it is the unique one in this particular case,
see \cite{DM})
and $P$ is a smooth point of $E_\alpha$ which is not a critical point. (In this case none
smooth point of $E_\alpha$ is critical). In the same way one has that the strict
transform of $\delta'$ and $y_1'$ by $\sigma$ intersect in a point $P'\in
E_\alpha\subset E$. As $a\neq a'$ we have $P\neq P'$ and it is known that
$[\delta, \delta '] =  I_0(\delta, x)I_0(\delta', y)=  I_0(\delta,
y)I_0(\delta', x)$.
\end{proof}

\medskip

{\bf c-2)}\underline{\it Case  $a= a'$.}

\medskip

Following the notations of the previous Lemma, as $\gcd(p,q)=1$, the fibres
$y_1$ and $y_1'$ are irreducible, so the case $a=a'$ is equivalent to say $P=P'$
and  $[\delta, \delta '] >  I_0(\delta, x)I_0(\delta', y)$. In this case one has
to  iterate the process.

\begin{remark}
In all the described cases but in case c-2) we have proved that
$$
[\delta, \delta'] = \min \{I_0(\delta, x)I_0(\delta', y),
I_0(\delta,y)I_0(\delta', x)\}\; .
$$
On the other hand, in the case c-2) for the first iterated pencils
$(\Lambda_1,y_1,q_1/p_1)$ and $(\Lambda_1', y_1', q'_{1}/p'_{1})$ of $\delta$ and
$\delta'$ respectively we have $\Lambda_1=\Lambda_1'$,
$y_{1}=y'_1$ and $P=P'$. Moreover, for the same reasons as in {\bf c-1}, the divisor $E_{\alpha}$ is a dicritical divisor of
$\Lambda_1$ and $P\in E_{\alpha}$ is not a critical point.
\end{remark}

\subsubsection{General recursive step}\label{recursive step}

For an index $i\geq 1$,
let $(\Lambda_j, y_j, q_j/p_j)$ and
$(\Lambda'_j, y'_j, q'_j/p'_j)$, $j\le i$, be the sequence of iterated pencils for
$\delta$ and $\delta'$.
Let us assume that, for $j\leq i$ one has
$\Lambda_j=\Lambda_j'$ and  $y_j=y'_j$ and moreover $q_j/p_j=q_j'/p_j'$ for $j< i$.

Let $\pi : (X,E)\to (\C^2,0)$ be the minimal resolution of the pencil
$\Lambda_i=\langle x_i,y_i\rangle$. Let $\widetilde \delta$ (resp.
$\widetilde \delta '$) be
the strict transform of $\delta $ (resp. $\delta '$) by $\pi$.
We suppose that $\widetilde \delta\cap E= \widetilde \delta '\cap
E$ and we denote $Q$ this point.

Following the same idea of the proof of Theorem 2 in \cite{TopIm}, we assume the
following properties:

\begin{enumerate}
\item $Q = \widetilde \delta\cap E= \widetilde \delta '\cap E$ is a smooth point of
the exceptional divisor $E$. So, there exists a (unique)
irreducible component $E_{\alpha}\subset E$ such that
$Q = \widetilde \delta\cap E_\alpha= \widetilde \delta '\cap E_\alpha$.
\item
The irreducible component $E_\alpha$ is dicritical for the
pencil $\Lambda_i$.
\item
There exists
a (unique) branch $\zeta_i$ of $\lambda_i = \{ y_i=0\}$ such that
it is a curvette at the point $Q$ (i.e. 
its strict transform by $\pi$ is smooth and transversal to $E_{\alpha}$  at the point $Q$).
\end{enumerate}

Let $\sigma : (X', E')\to (\C^2,0)$ be the composition of $\pi$
with the minimal modification of $(X, Q)$
until the strict transform of $\delta$ and $\delta '$  by
$\sigma$ meets
the exceptional locus $E'$ at  smooth points, $P, P'$, and also such  that the
strict transforms of $\zeta_i$ by $\sigma$ do not meet $P$ nor $P'$.

\begin{proposition} \label{IntersectionMult}
Following the above notations we obtain :
\begin{enumerate}
\item
If $P\neq P'$ one has :
$$
[\delta ,\delta '] = \min \{  [E_\alpha, \widetilde \delta]_Q \; [\delta '
,\zeta_i], [E_\alpha, \widetilde \delta ']_Q\; [\delta ,\zeta_i]\}\; .
$$
\item
If $P= P'$ then $q_i/p_i=q'_i/p'_i$, $\Lambda_{i+1}=\Lambda'_{i+1}$,
$y_{i+1}=y'_{i+1}$ and we proceed with a new iteration.
Moreover one has $P=P'$ if and only if
$[\delta ,\delta ']>  [E_\alpha, \widetilde \delta]_Q \; [\delta ' ,\zeta_i] =
[E_\alpha, \widetilde \delta ']_Q \; [\delta, \zeta_i]$.
\end{enumerate}
\end{proposition}

\begin{remark}\label{rkE&e}
In \cite{TopIm} the next property is also included in the above list, as fourth property, for the recursive step: 

$\ll$
Let $m_i = [\delta,\zeta_i]$, the set
$\{\seq{m}{-1}{i-1}\}$ contains all the maximal contact values  of $\delta$
smaller than $m_i$. $\gg$

In our case one can add the corresponding one for 
$\delta'$, $m'_i = [\delta',\zeta_i]$. 
(See also the proof of Theorem \ref{thm1} for the definition of the integers $m_i$.)
Here we do not need it because the semigroup
of each branch is already computed in Theorem \ref{thm1}.
However it is interesting to notice that (see \cite{TopIm}) :\\
$[E_\alpha, \widetilde \delta]_Q = \gcd (m_{-1}, m_0, \ldots , m_{i-1}):=
e_{i-1}$ and so, for $P\neq P'$: 
$$[\delta ,\delta '] =\min \{ e_{i-1}
m_i', e'_{i-1}m_i \}
$$ 
\end{remark}

\medskip
The next lemma will be useful for the proof of Proposition \ref{IntersectionMult}.

\begin{lemma}\label{lemme3} With the above notations, we have the following
equivalence:
$$
\displaystyle\frac{q'_i}{p'_i}
>\frac{q_i}{p_i}   \Longleftrightarrow \frac{[\widetilde \delta ' ,\widetilde
\zeta_i]_Q}{[E_\alpha, \widetilde\delta ' ]_Q}>\frac{[\widetilde \delta  ,\widetilde
\zeta_i]_Q}{[E_\alpha, \widetilde\delta  ]_Q}.
$$
  \end{lemma}

\begin{proof}
For each irreducible component $\zeta $ of $\lambda _i$ different from $\zeta_i$ one
has, from lemma \ref{lemmaInterDiv}, $[\delta , \zeta]= [E_\alpha, \widetilde
\delta ]_Q\nu_\alpha (\zeta) $ and $[\delta ', \zeta ]= [E_\alpha, \widetilde
\delta ' ]_Q\nu_\alpha (\zeta) $
and so $$[\delta , \zeta ]/[E_\alpha, \widetilde \delta ]_Q=[\delta ',
\zeta ]/[E_\alpha, \widetilde \delta ' ]_Q .$$
As $$ \displaystyle\frac{q_i}{p_i}  = \frac{[\delta ,\lambda_i ]}{I_0(\delta
,x_i)}=  \frac{[\delta ,\zeta_i]}{I_0(\delta ,x_i)}+\displaystyle\sum_{{\zeta
\in \lambda_i}, \ {\zeta\neq \zeta_i}} \frac{[\delta , \zeta]}{[E_\alpha,
\widetilde \delta ]_Q \; \nu_\alpha (x_i)}  
$$ 
and 
$$ \displaystyle\frac{q'_i}{p'_i}  =
\frac{[\delta ' ,\lambda_i]}{I_0(\delta' ,x_i)}= \frac{[\delta '
,\zeta_i]}{I_0(\delta ',x_i)} +\displaystyle\sum_{{\zeta \in \lambda_i}, \
{\zeta\neq \zeta_i}} \frac{[\delta ', \zeta]}{[E_\alpha, \widetilde \delta ']_Q\; \nu_\alpha (x_i)}  ,$$ 
one has
\begin{align*}
\displaystyle\frac{q'_i}{p'_i} >\frac{q_i}{p_i}   & \Longleftrightarrow
\frac{[\delta ' ,\zeta_i]}{I_0(\delta ',x_i)}>  \frac{[\delta
,\zeta_i]}{I_0(\delta ,x_i)}
 \Longleftrightarrow \frac{[\delta ' ,\zeta_i]}{q_0\ldots q_{i-1} I_0(\delta
',x)}>  \frac{[\delta ,\zeta_i]}{q_0\ldots q_{i-1} I_0(\delta ,x)} \\
& \Longleftrightarrow \frac{[\widetilde\delta ' ,\widetilde\zeta_i]_Q+ [E_\alpha,
\widetilde \delta ' ]_Q\; \nu_\alpha(\zeta_i)}{ [E_\alpha, \widetilde \delta ' ]_Q }>
\frac{[\widetilde\delta ,\widetilde\zeta_i]_Q+ [E_\alpha, \widetilde
\delta]_Q\; \nu_\alpha(\zeta_i)}{ [E_\alpha, \widetilde \delta]_Q} \\
& \Longleftrightarrow
\frac{[\widetilde \delta ' ,\widetilde \zeta_i]_Q}{[E_\alpha, \widetilde\delta '
]_Q}>\frac{[\widetilde \delta  ,\widetilde \zeta_i]_Q}{[E_\alpha, \widetilde\delta
]_Q}\; .
\end{align*}
           
\end{proof}

\begin{proof}[{\bf Proof of Proposition \ref{IntersectionMult}}] \

{\bf 1)} Let us assume $P\neq P'$.
We start by computing $[\widetilde\delta , \widetilde\delta ']_{Q}$. Taking into
account that $\widetilde \zeta_i$ and $E_\alpha$ are both smooth and intersect
transversaly at $Q$,
we can choose a pair of local coordinates $(u,v)$ at $(X,Q)$ in such a way that
$u=0$ and $v=0$ are the (local) equations of $E_\alpha$ and
$\widetilde \zeta_i$ respectively.
Thus one has locally, at
$Q\in X$, exactly the same situation as in the First step \ref{first step} for the pencil
$\langle u, v\rangle$ instead $\Lambda = \langle x, y\rangle$.
Thus, we have:
$$
P\neq P'  \Longleftrightarrow [\widetilde\delta ,\widetilde\delta ']_Q =
\min \{[E_\alpha , \widetilde \delta ]_Q[\widetilde\delta ' ,\widetilde\zeta_i]_Q, [E_\alpha ,
\widetilde \delta '  ]_Q[\widetilde\delta ,\widetilde\zeta_i]_Q\}\, .
$$
Now one has to compute  $[\delta ,\delta '] $. By lemma \ref{lemmaInterDiv} we
have :
$$[\delta ,\delta '] =[\widetilde \delta , \widetilde \delta ']_Q+  [E_\alpha ,
\widetilde \delta ' ]_Q\; \nu_\alpha (\delta) =[\widetilde \delta , \widetilde \delta']_Q
+  [E_\alpha , \widetilde \delta  ]_Q\; \nu_\alpha (\delta ') .$$

Let assume that $[E_\alpha ,\widetilde \delta']_Q[\widetilde\delta
,\widetilde\zeta_i]_Q < [E_\alpha , \widetilde \delta   ]_Q[\widetilde\delta '
,\widetilde\zeta_i]_Q$,
(from lemma \ref{lemme3} this is equivalent to $q_i'/p_i'>q_i/p_i$) then
$$
\begin{array}{cclr}
[\delta ,\delta ']& =&[E_\alpha , \widetilde \delta '  ]_Q[\widetilde\delta
,\widetilde\zeta_i]_Q +  [E_\alpha , \widetilde \delta ' ]_Q\nu_\alpha (\delta)&\\
&= & [E_\alpha , \widetilde \delta ' ]_Q \left([\widetilde\delta ,\widetilde\zeta_i]_Q
+ \nu_\alpha (\delta) \right)&\\
&=& [E_\alpha , \widetilde \delta '  ]_Q[\delta ,\zeta_i] <
[E_\alpha , \widetilde \delta ]_Q[\delta' ,\zeta_i]  .  &
\end{array}
$$
 Notice that in case of $q_i/p_i=q_i'/p_i'$ one has $[\delta ,\delta
']=[E_\alpha , \widetilde \delta '  ]_Q[\delta ,\zeta_i]= [E_\alpha , \widetilde
\delta   ]_Q[\delta ' ,\zeta_i]$.

\medskip

{\bf 2)} Let $P=P'\in E_\beta\subset E'$. The strict transform of $\zeta_i$ by
$\sigma$ does not intersects $E_\beta$ at $P$ and obviously the same is true for all
the branches of $\lambda_i = \{y_i=0\}$.
So by Lemma \ref{lemmaInterDiv} we have
$[\lambda_i,\delta]= [E_\beta, \w \delta]_P\; \nu_\beta(y_i)$ and
$[\lambda_i,\delta']= [E_\beta, \w \delta']_P\; \nu_\beta(y_i)$.
The same equalities are also true for $x_i$ instead $y_i$.
Then
$$
\frac{q_i}{p_i}= \frac{[\lambda_i,\delta]}{I_0(x_i,\delta)} =
\frac{\nu_\beta(y_i)}{\nu_\beta(x_i)} =
\frac{[\lambda_i,\delta']}{I_0(x_i,\delta')} =
\frac{q'_i}{p'_i}
$$
and as a consequence $\Lambda_{i+1} = \Lambda'_{i+1}$.
As $\nu_\beta(x_i^{q_i})=\nu_\beta(y_i^{p_i})$, the lifting of the meromorphic
function $\psi  = y_i^{p_i}/x_i^{q_i}$ is not equal to zero along $E_\beta$ and has a
zero at the point $P$, thus $E_\beta$ is a dicritical divisor of $\Lambda_{i+1}$.
As in the proof of Lemma 1 of \cite{TopIm} one can see that $Q$ is not a critical
point of $\Lambda_{i+1}$ and so one has also that there exists an unique fibre
$y_{i+1} = y_i^{p_i}- a x_i^{q_i}\in \Lambda_{i+1}$ ($a = \w \psi(P) \in \C^*$)  such
that $I_0(y_{i+1}, \delta)> I_0(y_i^{p_i},\delta)$ and
$I_0(y_{i+1}, \delta')> I_0(y_i^{p_i},\delta')$.
The fact that $E_\beta$ is dicritical and $P\in E_\beta$ is not a critical point
implies that there exists an unique branch $\zeta_{i+1}$ of
$\lambda_{i+1}=\{y_{i+1}=0\}$ such that its strict transform is a curvette at the
point $P$.

Notice that, if $\sigma$ is not the minimal resolution of $\Lambda_{i+1}$ then the
new blowing-ups up to reach it do not affect the described situation on
$P\in E_\beta$. 

 Notice also that in this case we have $[\delta ,\delta
']>[E_\alpha , \widetilde \delta '  ]_Q[\delta ,\zeta_i]= [E_\alpha , \widetilde
\delta   ]_Q[\delta ' ,\zeta_i]$.
\end{proof}

\medskip
\begin{remark}\label{rmk_a}
Let us summarize the case when $P\neq P'$ and $q_i/p_i=q_i'/p_i'$.
Note that at this step, $q_i/p_i=q_i'/p_i'$ is equivalent to say
$\Lambda_{i+1}=\Lambda_{i+1}'$. In this case $P$ and $P'$ belong to the same
dicritical component $E_\beta \subset E'$ of   the minimal resolution of
$\Lambda_{i+1}$.
Let
$y_{i+1}= y_i^{p_i}-ax_i^{q_i}$, $\lambda_a=\{y_{i+1}=0\}$
 (resp.  $y'_{i+1} = y_i^{p_i}-a'x_i^{q_i}$,
$\lambda_{a'}=\{y'_{i+1}=0\}$)
be the unique fibre such that
$I_0(\delta , \lambda_a )> I_0(\delta , \lambda)$ (resp.
$[\delta' , \lambda_{a'} ]> [\delta' , \lambda ]$), for any other fibre $\lambda
$ of $\Lambda_{i+1}$. Thus we have two possibilities:

Either $a\neq a'$, and in this case $y_{i+1}\neq y'_{i+1}$ and
$\Lambda_{i+2}\neq \Lambda_{i+2} '$.

Otherwise $a=a'$. In this case there exists $\zeta , \zeta '$ irreducible components
of the same fibre
$\lambda_a$ such that the strict transform $\w{\zeta}$ of $\zeta$ (resp. $\w{\zeta'}$
of $\zeta'$) meets $E_{\beta}$ at $P$ (resp. at $P'$).
By using lemma \ref{lemmaInterDiv} it is easy to
check that it is equivalent to
$$
\frac{[\delta  , \zeta]}{[E_\beta,\widetilde\delta  ]}>\frac{[ \delta ' ,
\zeta]}{[E_\beta,\widetilde\delta ' ]} \left(\mbox{ resp. }  \frac{[\delta ' ,
\zeta']}{[E_\beta, \widetilde\delta ' ]}>\frac{[ \delta , \zeta
']}{[E_\beta,\widetilde\delta  ]} \right) \; .\ \ \ \ \ (*)
$$
As a consequence, always  using lemma   \ref{lemmaInterDiv} we have
$[\delta ,\delta '] = \nu_{\beta}(\delta) [E_\beta, \tilde\delta ' ]$
and as $\tilde\delta \cap \tilde \zeta '=\emptyset$ and $\zeta '$ is transversal to $E_\beta$, $\nu_{\beta}(\delta)=[\delta  ,\zeta ']$ and  we obtain
$$
[\delta ,\delta '] = [E_\beta, \widetilde \delta '][\delta  ,\zeta '] \; .
$$

In the same way we have $[\delta ,\delta '] = \nu_{\beta}(\delta ') [E_\beta,
\tilde\delta  ] = [E_\beta, \widetilde \delta   ][\delta ' ,\zeta ]$.

\end{remark}

\medskip

From the Proposition \ref{IntersectionMult} and its proof it is obvious that
the maximal integer $\kappa$ such that the requirements of the General recursive step are
satisfied can be determined in a more easy way: let
${\cal P}(x,y, \delta) = \{(\Lambda_i,
y_i, q_i/p_i) : i\ge 0\}$
(resp.
${\cal P}(x,y, \delta') = \{(\Lambda'_i,
y'_i, q'_i/p'_i) : i\ge 0 \}$ )
be the sequence of iterated pencils for $\delta$ (resp. for $\delta'$).
Then one has:

\begin{statement}\label{stat}
Let $\kappa$ be the largest integer such that, if
$\pi: (X,E)\to (\C^2,0)$ is the minimal resolution of the pencil $\Lambda_\kappa$, then
the  strict transforms of $\delta$ and $\delta'$ intersects $E$ at the same point
$Q\in E_\alpha\subset E$. Then $E_{\alpha}$ is a dicritical component for
$\Lambda_\kappa$ and $Q$ is not a critical point of $\Lambda_\kappa$.

Moreover,
if $\zeta$ is the unique branch of $\lambda_\kappa = \{y_\kappa=0\}$ such that
its strict transform by $\pi$ intersects $E$ at the point $Q$ one has:
$$
[\delta ,\delta '] = \min \{ [E_\alpha, \widetilde \delta]_Q \; [\delta'
,\zeta], [E_\alpha, \widetilde \delta ']_Q\; [\delta ,\zeta]\}\; .
$$
\end{statement}

\begin{proof}
The proof goes by induction on $\kappa$, the case $\kappa=0$ is just
the first step \ref{first step} and the inductive step is the General recursive
step
\ref{recursive step}.
\end{proof}

Although the above statement completely solves the problem of determining the
intersection multiplicity of $\delta$ and $\delta'$, the computation of $\kappa$ requires
the resolution of the iterated pencils. Let us see that this computation can be set
up in another way.

Let $\{m_i, i\ge -1\}$ and $\{m'_i, i\ge -1\}$ be the sequences defined in the proof
of Theorem \ref{thm1} and Remark \ref{rmk_m} for $\delta$ and $\delta'$. For $i\ge
-1$, let $e_i= \gcd (\seq{m}{-1}i)$ and $e'_i= \gcd (\seq{m'}{-1}i)$.

\begin{theorem}\label{mainth-plane}
Let $\ell$ be the smallest integer such that
there exists a branch $\zeta$ of the fibre $\lambda_\ell = \{y_\ell =0\}$ of
$\Lambda_{\ell}$ such that
$[\delta, \zeta]/I_0(\delta,x)\neq [\delta', \zeta]/I_0(\delta',x)$. Then:
$$
[\delta, \delta'] = \min \{ e_{\ell-1} [\delta',\zeta],
e'_{\ell-1} [\delta,\zeta]\}\; .
$$
\end{theorem}

\begin{proof}
Let $\kappa$ be the integer defined in Statement \ref{stat}, then one has that
$[\delta, \zeta]/e_{i-1}= [\delta', \zeta]/e'_{i-1}$ for all branches
$\zeta$ of the fibres $\{y_i=0\}$ for $i< \kappa$.
Let $\ell$ be equal to $\kappa+1$ if we are in the case $a=a'$ described in Remark
\ref{rmk_a} and let $\ell=\kappa$ otherwise.
Proposition \ref{IntersectionMult} and Remark \ref{rmk_a} implies that
there exists a branch $\zeta$ of $\lambda_\ell=\{y_\ell=0\}$ such that
$[\delta, \zeta]/e_{i-1}\neq [\delta', \zeta]/e'_{i-1}$.
Then, the integer $\ell$ could be defined as the smallest such that
$[\delta, \zeta]/e_{i-1}\neq[\delta', \zeta]/e'_{i-1}$
for some branch $\zeta$ of $\lambda_\ell$.
The same results implies the stated equality for
$[\delta,\delta']$.

Now, it is well known that, in our conditions, $d_i= e_{i-1}/e_i =
e'_{i-1}/e'_i=d'_i$ for $i\le \ell-1$ (see e.g. \cite{D-Lille}). Moreover,
$e_{-1} = m_{-1} = I_0(\delta, x)$ and
$e'_{-1} = m'_{-1} = I_0(\delta', x)$.
Then for the integer $\ell$ we have just defined above
one has that also $\ell$ is  the smallest integer such that
$[\delta, \zeta]/I_0(\delta, x) \neq [\delta', \zeta]/I_0(\delta',x)$.
\end{proof}        

\begin{remark}
Notice that the formula for $[\delta,\delta']$ in the above Theorem is not the
same as in Remark \ref{rkE&e}. Let us assume that we are in the case $\ell=\kappa+1$,
then one has that $[\delta, \delta']=e_{\kappa-1} m'_\kappa
=e'_{\kappa-1}m_\kappa$ and also
$[\delta, \delta'] = \min \{ e_{\ell-1} [\delta',\zeta],
e'_{\ell-1} [\delta,\zeta]\}$.
However $m'_{\ell} \neq [\delta',\zeta]$ because $\zeta$ does not go by $P'$. In
fact it could happen even that
$m'_\ell/e'_{\ell-1} = m_{\ell}/e_{\ell-1}$ (see the Example below). Thus in order to
detect the integer $\ell$ (and so $\kappa$) it does not suffice with the sequences
$\{m_i\}$ and $\{m'_i\}$.
\end{remark}

\begin{remark}
If there exists an integer $r$ such that
$\Lambda_r\neq \Lambda_r'$ then there
exists $\zeta$, and $\zeta '$ branches of $\lambda_r = \{y_r=0\}$ and
$\lambda_r'=\{y'_r=0\}$ such that
$\widetilde \delta \cap \widetilde \zeta \neq \emptyset$ and
$\widetilde \delta ' \cap \widetilde \zeta'\neq \emptyset$
respectively. Then, one has:
$$
[\delta ,\delta '] = e'_{r-1}[\delta  ,\zeta
']=e_{r-1} [\delta ' ,\zeta ]\; .
$$

Note that the only condition is that $\Lambda_r\neq\Lambda_r '$. In particular the
branches $\delta, \delta'$ could be separated before the step $r$.
\end{remark}

\begin{example}\label{ex1}
The following example from \cite{TopIm} shows that the use we made of the branches of
the fibres and not only of the fibres themselves is required.
Let
$\delta = \{4y^2-4x^3-4yx^3+x^6=0\}$,
$\delta' = \{4y^2-4x^3+4yx^3+x^6=0\}$.
The Puiseux expansion of them are
$x = t^2, y =t^3+ (1/2)t^6$ and
$x = t^2, y =t^3 - (1/2)t^6$ respectively. Their intersection multiplicity
$[\delta, \delta ']$ is equal to 9.
As one can prove in an easy way the sequence of pencils and fibres for both are the
same (as well as the sequences
$\{m_i\}$ and $\{m'_i\}$). For the first steps one
founds:
\begin{eqnarray*}
y_1 &= &y^2-x^3  \\
y_2 &= &(y^2-x^3)^2-x^9  \\
y_3 &= &((y^2-x^3)^2-x^9)^6-(1/64)x^{63} \\
y_4 &= &(((y^2-x^3)^2-x^9)^6-(1/64)x^{63})^{42}-(3/256)^{42}(x^{63})^{43}\; .
\end{eqnarray*}
and the sequences of $m_i =m_i'$ is $(m_{-1}=2,m_0=3, m_1=9,m_2=12,
m_3=15)$ and
$(q_0/p_0= q_0'/p_0'= 3/2, q_1/p_1= q_1'/p_1'= 3/2 , q_2/p_2= q_2/p_2'= 7/6, \ldots  )$.
Note that $m_{-1}=2$ and $m_0=3$ are enough to compute the semigroup of
$\delta$ (see Corollary \ref{cor1}):  the one generated by $2,3$. The same is true
for
$\delta'$.

The fibre $y_1$ is irreducible. But at the second step, the fibre $y_2$ has two
branches $\zeta_1,\zeta_2$ such that  $[\delta, \zeta_1]= 12$ and $[\delta,
\zeta_2]= 9$ and inversely $[\delta', \zeta_1]= 9$ and $[\delta', \zeta_2]= 12$.
Thus following Theorem \ref{mainth-plane} we have that $\ell=2$
and
$$
[\delta,\delta'] = \min\{ e_{1}[\delta',\zeta], e'_1 [\delta,\zeta']\} =
\min\{ 1 \cdot 9, 1\cdot 9\} = 9\;.
$$

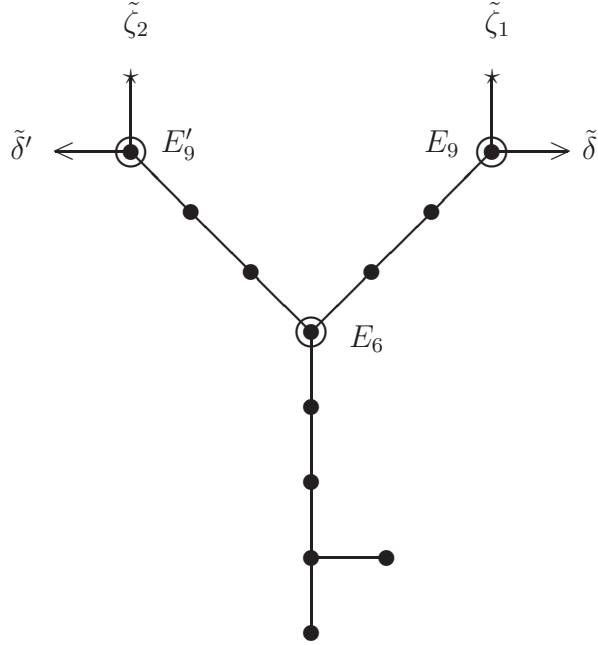
\begin{figure}[h]
$$
\unitlength=1.00mm
\begin{picture}(30.00,90.00)
\thicklines
\put(20,10){\line(0,1){40}}
\put(20,10){\circle*{2}}
\put(20,20){\circle*{2}}
\put(20,30){\circle*{2}}
\put(20,40){\circle*{2}}
\put(20,50){\circle*{2}}
\put(20,50){\circle{4}}
\put(25,48){$E_6$}

\put(20,20){\line(1,0){10}}
\put(30,20){\circle*{2}}

\put(20,50){\line(-1,1){24}}
\put(20,50){\line(1,1){24}}

\put(28,58){\circle*{2}}
\put(36,66){\circle*{2}}
\put(44,74){\circle*{2}}
\put(44,74){\circle{4}}
\put(35,74){$E_9$}

\put(12,58){\circle*{2}}
\put(4,66){\circle*{2}}
\put(-4,74){\circle*{2}}
\put(-4,74){\circle{4}}
\put(0,74){$E'_9$}

\put(44,74){\line(1,0){10}}
\put(44,74){\line(0,1){10}}
\put(43,83){$\star$}
\put(51.5,73){$>$}
\put(43,90){$\tilde \zeta_1$}
\put(56,73){$\tilde \delta$}

\put(-4,74){\line(-1,0){10}}
\put(-4,74){\line(0,1){10}}
\put(-5,83){$\star$}
\put(-14.5,73){$<$}
\put(-5,90){$\tilde \zeta_2$}
\put(-20,73){$\tilde \delta '$}
\end{picture}
$$
\caption{Graph of the minimal resolution of $\delta \cup \delta'\cup \{y_2=0\}$
}\label{fig1}
\end{figure}

In such a way, as moreover  $[\delta, \delta ']=9$, at the sixth blow-up
(the
corresponding divisor $E_6$ is the dicritical one of the pencil $\Lambda_2$, see Figure \ref{fig1})
the strict transforms of $\delta $ and $\delta ' $ separate, as do the ones of
$\zeta_1$ and $\zeta_2$, while the stricts transforms of $\delta$ and  $\zeta_1$,
resp. of $\delta'$ and  $\zeta_2$, separate three blow-ups later (the
corresponding divisors $E_9$ and $E'_9$ are dicritical for the pencil $\Lambda_3$).

With the above notations it means that we have $P\neq P'$ at $E_6$
although the pencils of $\delta $ and $\delta'$ are the same.

The resolution graph of the minimal resolution of $\delta \cup \delta'\cup
\{y_2=0\}$ is depicted in Figure \ref{fig1}.
\end{example}

\subsection{Case of a normal singularity $(Z,z)$}

Let $\gamma, \gamma'$ be two irreducible curves in $(Z,z)$ and $\varphi
(\gamma)=\delta $ and $\varphi (\gamma ')=\delta' $ their image in the plane $\C^2$.
We will assume that $\delta  \neq \delta'$.
Recall that $\varphi_* \gamma =k\delta$ and $\varphi _* (\gamma ')=k'\delta ' $
for some positive integers $k,k'$.

Let $\PP(f,g,\gamma) = \{(\Phi_i,g_i,q_i/p_i) : i\ge 0\}$
and
$\PP(f,g,\gamma') = \{(\Phi'_i,g'_i,q'_i/p'_i) : i\ge 0\}$ be the iterated sequence
of pencils of the branches $\gamma$ and $\gamma'$ w.r.t. $\varphi=(f,g)$.

\begin{theorem}\label{TheoremPrincipal}
Let $\ell$ be the smallest integer such that
there exists a branch $\rho$ of the fibre $\xi = \{g_\ell =0\}$ of
$\Phi_{\ell}$ such that
$$
\displaystyle\frac{I_z(\gamma , \varphi ^* (\varphi (\rho)))}{I_z(\gamma,f)}\neq
\displaystyle\frac{I_z(\gamma ', \varphi ^* (\varphi (\rho)))}{I_z(\gamma',f)}\; .
$$
Then:
$$
[\delta, \delta ']= min \left\{
  \displaystyle\frac{e_{\ell-1}'I_z(\gamma ,
\varphi ^* (\varphi (\rho)))}{k} ,
\displaystyle\frac{e_{\ell-1}I_z(\gamma ',
\varphi ^* (\varphi (\rho)))}{k'}
\right\}\; .
$$
\end{theorem}

\begin{proof}
Let
${\cal P}(x,y, \delta) = \{(\Lambda_i,
y_i, q_i/p_i) : i\ge 0\}$
(resp.
${\cal P}(x,y, \delta') = \{(\Lambda'_i,
y'_i, q'_i/p'_i) : i\ge 0 \}$ )
be the sequence of iterated pencils for $\delta$ (resp. for $\delta'$).
As a consequence of Proposition \ref{proj-for}, for $j\ge 0$,
$(\varphi^* (\Lambda_j), \varphi^* y_j,  q_j/p_j)= (\Phi_j, g_j,q_j/p_j)$.

Let $\zeta\subset (\C^2,0)$ be a plane branch. By the projection formula
\ref{proj-for} one has $[\delta, \zeta ] = I_z(\gamma,\varphi^*\zeta)/k$ and so
$$
\frac{[\delta, \zeta ]}{I_0(\delta,x)}
= \frac{I_z(\gamma,\varphi^*\zeta)}{I_z(\gamma,f)}\; .
$$
This equalities applied to the branches $\zeta$ of $\lambda_i =\{y_i=0\}$ and both
pairs
$(\gamma,\delta)$, $(\gamma',\delta')$ implies that
the integer $\ell$ defined in Theorem \ref{mainth-plane} is charaterized as the
one described in the statement.

Now, the equality for the intersection multiplicity of $\delta$ and $\delta'$ given
in Theorem \ref{mainth-plane} is (by using the projection formula) the one described
in the statement of the Theorem.
\end{proof}

\begin{remark}
Using remark \ref{rmk_m}, the expression of the formula of Theorem
\ref{TheoremPrincipal} in terms of intersections multiplicities is the following one :
$$
[\delta, \delta ']= \frac{1}{k k'}\min  \left\{
\epsilon'_{\ell-1}
I_z(\gamma ,\varphi ^* (\varphi (\rho))),
\epsilon_{\ell-1}
I_z(\gamma ',\varphi ^* (\varphi (\rho))) \right\}.
$$
\end{remark}

\section{The discriminant case}

The results proved in \cite{PNS}
about the behaviour of the critical
locus  of the map $\varphi$ and its relation with the special
fibres of the pencil
$\Phi = \langle f,g \rangle$ (see also \cite{DM} for the
plane case) allows to determine the topology of the irreducible
components of the discriminant curve, $ D(\varphi)$, of the morphism
$\varphi$.
  We recall that the discriminant curve is the image by $\varphi$ of the
critical locus, $C(\varphi)$,
of the map $\varphi$ (we consider only the one-dimensional components
of the critical locus, so $C(\varphi)$ is the
topological closure of the vanishing
of the restriction of the jacobian determinant of $(f,g)$ to $(Z,z)\backslash
\{z\}$).
Moreover, the sequence of rational numbers
$q_i/p_i$, $i\ge 0$, for each branch of $C(\varphi)$ can be computed directly from
the pencils without the knowledge of the concrete branches of $C(\varphi)$ (see
\cite{PNS}). We will show here how we obtain the whole topology type of the
discriminant curve using only the construction of the pencils.

To make this paragraph self-contained, we recall here the construction of  $\cal S$
(the set of sequences $( B_i)_{i\geq 0}$ associated to each branch of the
discriminant curve) in a similar way to the one given in section 4 of
\cite{TopIm}.

\begin{setting}

Let $\pi: (X,E)\to (Z,z)$ be the minimal good resolution of $(Z,z)$ which is also a
resolution of the pencil
$\langle g,f \rangle$ and of the curve $\{fg=0\}$ (i.e. the minimal good resolution
of the singularity $(Z,z)$ such that
$(g/f)\circ \pi$ is a morphism  and the strict transform of
$\{fg=0\}$ is smooth and transversal to the exceptional locus, see \cite{PNS}).

An irreducible component $E_\alpha $ of the exceptional divisor $E$
is  called a {\it nodal} component
(called rupture component in
\cite{TopIm})
if either $E_\alpha$ is non-rational or 
the number of connected components of $\bar C_{red}\backslash E_\alpha$ is
at
least three, where $\bar C_{red}$ 
stands for the total transform  of $C= \{ fg=0\}$ by $\pi$ with reduced structure.
A {\it nodal zone} $R \subset E$ is a maximal connected union
of irreducible exceptional components containing at least one
nodal component and such that the ratio
$\frac{\nu_\alpha(g) }{ \nu_\alpha (f)}$ is constant for
$E_\alpha\subset R$. The term $r$-nodal zone is used to indicate that
$\frac{\nu_\alpha(g) }{ \nu_\alpha (f)}=r$.

Let ${\cal RZ}$ be the set of all nodal zones $R$ such that
$Q(R):=\frac{\nu_\alpha (g)}{\nu_\alpha (f)}\neq 1, \ E_\alpha
\subset R$. Let $E^{(1)}=\displaystyle\cup_{\alpha} E_\alpha
\subset E$ be the union of the irreducible components
$E_{\alpha}$ of
$E$ such that $\frac{\nu_\alpha (g)}{\nu_\alpha(f)}=1$.
Notice that the union $\cal D$ of dicritical
components of $E$ is included in $E^{(1)}$.
Let ${\cal A}$ be the set of
connected components of
$\overline{E^{(1)}\backslash {\cal D}}$ which contain a nodal component.
For
$A\in \cal A$ we define  $Q(A) :=\frac{\nu_\alpha (g)}{\nu_\alpha
(f)}=1$,
$E_\alpha \subset A$. Moreover let
${\cal P}_{\cal C}$ be the set of points  $P\in \cal D$ such
that either $P$ is a singular point of $\cal D$ or $P$  is a smooth
point of $\bar C_{red}$ which is a critical point of the map
$(g/f)\circ \pi_{|_{\cal D}}:  {\cal D}\longrightarrow \C\P^1$.
For $P\in {\cal P}_{\cal C}$ we put $Q(P):=1$.


\medskip
We denote ${\cal B}^0={\cal RZ}\cup {\cal A}\cup {\cal
P}_{\cal C}$
and fix $B\in {\cal B}^0$. We write
$Q(B)=\frac{q}{p} \in \Q$ with gcd$(p,q)=1$ and consider  the pencil
$\Phi _B =\langle  g^p, f^q\rangle$.
Notice that $\pi$ is also an embedded resolution of $\{
f^qg^p=0\}$.
Thus by theorem 3 of \cite{PNS} there exists a
unique special fibre $\xi _B=\{g^p -a f^q=0\}$ of  $\Phi _B$
whose strict transform $\tilde \xi _B$ by $\pi$ intersects $B$.

Let us denote
$\varphi_1=(f_1, g_1)=(f^q ,g^p -a f^q)$ and   $\sigma _B$ the minimal sequence of
blowing-ups of points such that $\pi_B:=\sigma_B\circ \pi$  is an embedded resolution
of $\{f_1g_1=0\}$.

One can construct the set  ${\cal B}^0 (\varphi _1)$ for
$\varphi _1$ in the same way we have defined ${\cal B}^0$ for
$(f,g)$ and for each $B^1\in {\cal B}^0(\varphi_1)$ we define $Q_1(B_1)$ as
$Q(B_1)$ with respect to
$g_1,f_1$. Moreover we add the symbol $B_\infty$ when there
exists a non reduced branch $r\zeta$ of $\{g_1=0\}$, $r>1$,
whose strict transform $\widetilde {r\zeta}$ by $\pi
_B$ intersects $\sigma _B^{-1}(B)$. When $B_\infty$ exists we put
$Q_1(B_\infty ):=\infty$.

Let us denote :
$$
{\cal B}^1_B = \{ B_1 \in {\cal B}^0(\varphi _1)  \ / \
\sigma _B (B^1)\subset B\}\cup \{ B_\infty\}
$$
and we set ${\cal B}^1= \displaystyle \cup _{B\in {\cal B}^0}{\cal B}^1_B$.
Moreover one has the map $\psi_1 : {\cal B}^1\longrightarrow
{\cal B}^0$ defined by $\psi_1(B_1)=B$ if and only if $B_1\in {\cal B}^1_B$.
Thus we can follow the same process for each element $B_1\in
{\cal B}^1$ with $B_1\neq B_\infty$ for any $ B\in {\cal B}^0$
and so we inductively construct the collection of sets
$\{ {\cal B}^i\}_{i\geq 0}$ together with maps
$\psi _i :  {\cal B}^i \longrightarrow {\cal B}^{i-1}$ and
$Q_i : {\cal B}^i\longrightarrow \Q\cup \{\infty\}$ in such a
way that the
definition of $\psi _i^{-1}(B)$ for $B \in  {\cal B}^{i-1}$
(and the map ${Q_i}_{| \psi _i^{-1}(B)}$) follows the same
process  described above for $B\in {\cal B}^0$.

In a such a way we obtain $\cal S$ the set of sequences $( B_i)_{i\geq 0}$ such
that $B_i\in {\cal B}^i$ and $\psi _i(B_i)=B_{i-1}$ for $i\geq
1$. We assume that if $B_i=(B_{i-1})_\infty$ then the sequence is
finite and ends at $B_i$.
For our purposes it is better to encode this information in the form of a
weighted oriented graph ${\cal T}$ in the obvious form. That is, one puts a vertex for each
$B\in{\cal B}^i$ ($i\ge 0$) with weight equal to
$Q_i(B)$. The set of oriented edges is the set of pairs
$(B_{i-1},B_i)\in {\cal B}^{i-1}\times {\cal B}^i$ such that
$\psi_i(B_i)=B_{i-1}$. Thus, the  graph ${\cal T}$ is the union of
$\#{\cal B}^0$ trees, each one with root the corresponding element (vertex) of
${\cal B}^0$. We will call the above defined graph the graph of
$\varphi$.
In this way, the set ${\cal S}$ corresponds to the maximal completely ordered (and
weighted) subtrees of ${\cal T}$.

\end{setting}

\begin{setting}[{\bf The branches of the discriminant}]

Let $\CC(\varphi)$ be the set of branches of $C(\varphi)$ which are not branches of
$fg=0$.
Let also denote now by $\Delta(\varphi)$ the set of branches of the discriminant
locus wich are the image of a branch of $\CC(\varphi)$, so the map $\varphi$ gives a
surjective map  $\varphi: \CC(\varphi)\to \Delta(\varphi)$.

As a consequence of the results of \cite{LMW}, \cite{Mi} and \cite{PNS} one
has that the set $\CC(\varphi)$ can be decomposed as
$\bigcup_{B\in {\cal B}^0} \JJ_B$, $\JJ_B\neq \emptyset$ for all $B\in \BB^0$, and
such that $\gamma\in  \JJ_B$ if and only if its strict transform by $\pi$ intersects
$B$. Moreover, for $\gamma\in \JJ_B$ one has
$I_z(g,\gamma)/I_z(f,\gamma)=Q(B)$. Note that $Q(B)=q/p$ (the rational number defined
in \ref{pencils}).

\begin{remark}
The decomposition described in \cite{PNS} relating the special values and
the branches of the critical locus $\CC(\varphi)$ is not the same as the one described here: in
fact it could happen that a special component (i.e. a connected component of
$E\setminus \DD$) can be decomposed in several different $r$-nodal zones.
On the other hand, the points of a divisorial component in $\PP_{\CC}$ offers a more
precise decomposition than the one given by the corresponding $1$-nodal zone described
in \cite{Mi}. In this sense the decomposition from the elements of $\BB$ is a mixture
of both decompositions.
\end{remark}

Now, by Theorem 4 and Corollary 3 in \cite{PNS}, the special fiber $\xi_B=\{g^p-a f^q=0\}$ of $\Phi=\langle g,f\rangle$ defined
above just coincides with the fiber $g_1$ defined in \ref{pencils}. Thus the pencil $\Phi_1=\langle f_1, g_1\rangle$ is
exactly the one corresponding to the map $\varphi_1 =(g_1,f_1)$. Moreover, the
critical locus $\CC(\varphi_1)$ is the same as $\CC(\varphi)$, the one of the map
$\varphi$ (this is a straighforward computation, see also the proof of Theorem 4 in
\cite{TopIm} or lemma 6.4 in \cite{GB-PP}).
That means that we can repeat the same process for the new map and so consider
$B_1\in \BB^{0}(\varphi_1)$ such that the strict transform of $\gamma$ by the
corresponding resolution map intersects $B_1$.

The above construction makes possible to define the map
$F: \CC(\varphi)\to \SS$ such that for
$\gamma\in \CC(\varphi)$, $F(\gamma)\in \SS$ is the unique sequence $(B_i)_{i\ge 0}$
such that the corresponding strict transform of $\gamma$ intersects $B_i$ for all
$i$. Note that the fact that $\JJ_{B_i}\neq \emptyset$ for all $i$ implies that the map $F$ is surjective.

\end{setting}

\begin{theorem}\label{Thm1-3}
Let $\gamma\in \CC(\varphi)$, $\delta = \varphi(\gamma)\in \Delta(\varphi)$ and 
let $F(\gamma)=(B_i)_{i\ge 0}$. Then
$(Q_i(B_i))_{i\ge 0}$ together with the integer $k$ such that
$\varphi_{*}\gamma= k \delta$
determines the semigroup
(and so the topological type) of the branch $\delta$.
\end{theorem}

\begin{proof}
One needs to compute the sequence $\{m_i: i\ge -1\}$ described in Theorem \ref{thm1}.
For $i\geq 0$ let $r_i=q_i/p_i = Q_i(B_i)$, $\gcd(q_i,p_i)=1$.
As $r_i=q_i/p_i= I_z(g_i,\gamma)/I_z(f_i,\gamma)$ and
$f_i=f_{i-1}^{q_{i-1}}=f^{q_0\ldots q_{i-1}}$ we deduce :
$$
I_z(g_i,\gamma)=r_i I_z(f_i,\gamma) = r_iq_{i-1}I_z(f_{i-1},\gamma)=
(\pr_{k=0}^{i-1} q_k ) r_iI_z(f,\gamma)\; .
$$

Moreover, using notations of remark \ref{rmk_m}, $\mu_{-1}=I_z(f,\gamma)$ and for
$i\ge 0$:
$$
\mu_i=  I_z(g_i,\gamma)= I_z(\varphi^* y_i,\gamma)=
I_0(y_i,\varphi_*\gamma)= k I_0( y_i,\delta)\; .
$$
Thus
$$
\mu_i/k = I_0(y_i, \delta) =  I_z(g_i,\gamma)/k
= ( \pr_{k=0}^{i-1} q_k) r_iI_z(f,\gamma) / k
=(\pr_{k=0}^{i-1} q_k ) r_i\mu_{-1} /k\;.
$$
So for  the elements $m_i, i\geq -1 $  of the semigroup of $\delta$ we obtain:
\begin{align*}
m_{i+1} &= d_i m_i+\frac{\mu_{i+1}}{k}-p_i\frac{\mu_i}{k}
=  
d_i m_i+
(\pr_{k=0}^{i} q_k ) r_{i+1} \frac{\mu_{-1}}{k}-p_i
(\pr_{k=0}^{i-1} q_k ) r_{i}\frac{ \mu_{-1}}{k} \\
&=
d_i m_i+(r_{i+1}-1)( \pr_{k=0}^{i}
q_k ) \frac{\mu_{-1}}{k}\; .
\end{align*}
(As in Remark \ref{rmk_m} $\epsilon_i=\gcd(\seq{\mu}{-1}i)$ and
$d_i=\epsilon_{i-1}/\epsilon_i$.)  
\end{proof}

\begin{remark}\label{rmk-critical}
Let $\SS$ be the set of maximal completely ordered subtrees of $\TT$. 
Then $\CC(\varphi) = \bigcup_{S\in \SS} F^{-1}(S)$ is a partition of the set of
branches $\CC(\varphi)$. It is clear that, if $\gamma, \gamma'\in F^{-1}(S)$ for
some $S\in \SS$, then $\varphi(\gamma)$ and $\varphi(\gamma')$ are equisingular: both have
the same semigroup. Notice that $F^{-1}(S)\neq \emptyset$ for any $S\in \SS$,
however it could happen that $\# F^{-1}(S) >1$ for some $S$.

Moreover, if $F(\gamma)=F(\gamma')$ one has that the iterated sequences of pencils
$\PP(f,g,\gamma) = \{(\Phi_i,g_i,q_i/p_i) : i\ge 0\}$
and
$\PP(f,g,\gamma') = \{(\Phi'_i,g'_i,q'_i/p'_i) : i\ge 0\}$ are the same but not at
the inverse, i.e. it could happen that $\PP(f,g,\gamma)=\PP(f,g,\gamma')$ but
$F(\gamma)\neq F(\gamma')$. A key point when $\gamma, \gamma'\in \CC(\varphi)$ is
that the iterated sequences of pencils can be determined without the previous
knowledge of the branches $\gamma$ and $\gamma'$ provided they are from the
critical locus.

For each $S = (B_i)_{i\ge 0} \in \SS$ let $Q(S)=(Q_i(B_i))_{i\ge 0}$. Theorem above implies that the set of sequences $\{Q(S) : S\in \SS\}$ theoretically permits to recover the set of topological types (the semigroups) of the branches of $\Delta(\varphi)$ (see the end of Remark \ref{rmk_m}). So in this sense both sets are equivalent. 
\end{remark}

\begin{setting}[{\bf The intersection multiplicity}]

Let $\gamma, \gamma' \in \CC(\varphi)$ and let
$\varphi_*(\gamma)=k \delta$,  $\varphi_*(\gamma')= k'\delta'$.
Assume that $\delta\neq \delta'$. The computation of the intersection multiplicity
$[\delta,\delta']$
of $\delta$ and $\delta'$ can be made as in the Theorem \ref{TheoremPrincipal}
and so it depends of a more detailled knowledge of the branches of the fibers.
\end{setting}

\begin{theorem}
Let $\ell$ be the smallest integer such that
there exists a branch $\rho$ of the fibre $\xi = \{g_\ell =0\}$ of
$\Phi_{\ell}$ such that
$$
\displaystyle\frac{I_z(\gamma , \varphi ^* (\varphi (\rho)))}{I_z(\gamma,f)}\neq
\displaystyle\frac{I_z(\gamma ', \varphi ^* (\varphi (\rho)))}{I_z(\gamma',f)}\; .
$$
Then:
$$
[\delta, \delta ']= min \left\{
  \displaystyle\frac{e_{\ell-1}'I_z(\gamma ,
\varphi ^* (\varphi (\rho)))}{k} ,
\displaystyle\frac{e_{\ell-1}I_z(\gamma ',
\varphi ^* (\varphi (\rho)))}{k'}
\right\}\; .
$$
\end{theorem}

\begin{remark}
Let $\gamma, \gamma'\in \CC(\varphi)$ be such that
$F(\gamma)=(B_i)_{i\ge 0} \neq F(\gamma')=(B'_i)_{i\ge 0}$ and let
$\ell$ be such that
$B_i=B'_i$ for $i< \ell$ and
$B_\ell\neq B'_\ell$. In this case there exists a branch
$\rho$ of $\xi = \{g_\ell=0\}$ such that its strict transform
intersects $B_{\ell}$ and does not $B'_\ell$. For such a branch one has that
$$
\displaystyle\frac{I_z(\gamma , \varphi ^* (\varphi (\rho)))}{I_z(\gamma,f)} >
\displaystyle\frac{I_z(\gamma ', \varphi ^* (\varphi (\rho)))}{I_z(\gamma',f)}\; .
$$
So, in this case one has that
$$
[\delta, \delta ']= \displaystyle\frac{e_{\ell-1}I_z(\gamma ',
\varphi ^*(\varphi (\rho)))}{k'}
\; .
$$

\end{remark}


\end{document}